\documentclass[11pt]{article}

\usepackage{amsfonts,amsmath,latexsym,color,epsfig,hyperref,mathdots,tikz}
\newtheorem{theorem}{Theorem}[section]
\newtheorem{lemma}{Lemma}[section]

\newtheorem{question}{Question}[section]
\newtheorem{claim}{Claim}[section]

\newenvironment{proof}
      {\medskip\noindent{\bf Proof:}\hspace{1mm}}
      {\hfill$\Box$\medskip}
\def\qed{\ifvmode\mbox{ }\else\unskip\fi\hskip 1em plus 10fill$\Box$}

\input{epsf}

\makeatletter
\def\Ddots{\mathinner{\mkern1mu\raise\p@
\vbox{\kern7\p@\hbox{.}}\mkern2mu
\raise4\p@\hbox{.}\mkern2mu\raise7\p@\hbox{.}\mkern1mu}}
\makeatother

\newcommand{\dbrace}[1]{
\begin{tikzpicture}
\draw (0,0) node[rotate=45]{$\underbrace{\rule{#1}{0pt}}$};
\end{tikzpicture}}

\title{\vspace{-0.7cm}Tower-type bounds for unavoidable patterns in words}
\author{David Conlon\thanks{Mathematical Institute, Oxford OX2 6GG,
United Kingdom. Email: {\tt david.conlon@maths.ox.ac.uk}. Research
supported by a Royal Society University Research Fellowship and by ERC Starting Grant 676632.}\and
Jacob Fox\thanks{Department of Mathematics, Stanford University, Stanford, CA 94305, USA. Email: {\tt jacobfox@stanford.edu}. Research supported by a Packard Fellowship, by NSF Career Award DMS-1352121 and by an Alfred P. Sloan Fellowship.}
\and
Benny Sudakov\thanks{Department of Mathematics, ETH, 8092 Zurich, Switzerland.
Email: {\tt benjamin.sudakov@math.ethz.ch}. Research supported in part by SNSF grant 200021-175573.}}
\date{}

\begin{document}
\maketitle

\begin{abstract}
A word $w$ is said to contain the pattern $P$ if there is a way to substitute a nonempty word for each letter in $P$ so that the resulting word is a subword of $w$. Bean, Ehrenfeucht and McNulty and, independently, Zimin characterised the patterns $P$ which are unavoidable, in the sense that any sufficiently long word over a fixed alphabet contains $P$. Zimin's characterisation says that a pattern is unavoidable if and only if it is contained in a Zimin word, where the Zimin words are defined by $Z_1 = x_1$ and $Z_n=Z_{n-1} x_n Z_{n-1}$.  We study the quantitative aspects of this theorem, obtaining essentially tight tower-type bounds for the function $f(n,q)$, the least integer such that any word of length $f(n, q)$ over an alphabet of size $q$ contains $Z_n$. When $n = 3$, the first non-trivial case, we determine $f(n,q)$ up to a constant factor, showing that $f(3,q) = \Theta(2^q q!)$.
\end{abstract}

\section{Introduction}

The term Ramsey theory refers to a broad range of deep results from various mathematical areas, like combinatorics, logic, geometry, ergodic theory, number theory and analysis, all connected by the fact that large systems contain unavoidable patterns. Examples of such results include Ramsey's theorem in graph theory, Szemer\'edi's theorem in number theory, Dvoretzky's theorem in asymptotic functional analysis and much more.

In this paper, we study the appearance of such unavoidable patterns in words, where words and patterns are here defined to be strings of characters from distinct fixed alphabets. We say that a word $w$ {\it contains} the pattern $P$ if there is a way to substitute nonempty words, which need not be disjoint or even distinct, for the letters in $P$ so that the resulting word is a subword of $w$, where a subword of $w$ is defined to be a string of consecutive letters from $w$.  Conversely, we say that $w$ {\it avoids} $P$ if $w$ does not contain $P$.

For example, it is a simple exercise to show that every four-letter word over a two-letter alphabet contains the pattern $xx$, while Thue~\cite{T06, T12} famously constructed an infinite word over a three-letter alphabet avoiding $xx$. This example alone has a surprisingly rich history~\cite{AS99, BP07}, being related, among other things, to work of Morse~\cite{M21} on symbolic dynamics.

For a positive integer $q$, we say that the pattern $P$ is {\it $q$-unavoidable} if every sufficiently long word over a $q$-letter alphabet contains a copy of $P$. In the example above, where $P = xx$, $P$ is $2$-unavoidable, but $3$-avoidable. We say that the pattern $P$ is {\it unavoidable} if it is $q$-unavoidable for all $q$. The unavoidable patterns were characterised by Bean, Ehrenfeucht and McNulty~\cite{BEM79} and, independently, by Zimin~\cite{Z82}. Zimin's characterisation, which is particularly appropriate for our purposes, says that a pattern is unavoidable if and only if it is contained in a Zimin word.

The {\it Zimin words} are defined recursively: $Z_1 = a$, $Z_2 = aba$, $Z_3 = abacaba$ and, in general, $Z_n = Z_{n-1}xZ_{n-1}$, where $x$ is a new letter. As well as playing a central role in the study of unavoidable patterns in words, these words are important in the study of Burnside-type problems, showing up in Ol'shanskii's proof of the Novikov--Adian theorem and, in a slightly different guise, in Zelmanov's work on the restricted Burnside problem (see~\cite{S14} for a thorough discussion).

It is natural and interesting to consider the quantitative aspects of Zimin's theorem. Following Cooper and Rorabaugh~\cite{CR14}, we let $f(n,q)$ denote the smallest integer such that every word of length $f(n,q)$ over an alphabet of size $q$ contains a copy of $Z_n$. It is a simple exercise to verify that $f(1, q) = 1$ and $f(2, q) = 2q + 1$. For general $n$, Zimin's work gives an Ackermann-type upper bound for $f(n, q)$. However, a combination of recent results due to Cooper and Rorabaugh~\cite{CR14} and Rytter and Shur~\cite{RS15} gives the considerably better bound that, for $n \geq 3$ and $q \geq 2$,
\[f(n,q) \leq q^{q^{\iddots^{q + o(q)}}}
 \raisebox{-10pt}[0pt][0pt]{\hspace*{-24pt}\dbrace{11pt}}
 \raisebox{0pt}[0pt][0pt]{\hspace*{-11pt}\scriptsize{$n$-$1$}}
,\]
where the $o(q)$ term in the topmost exponent does not depend on $n$ (in fact, it can be taken to be zero when $q$ is sufficiently large). 

Our first result is a lower bound matching the upper bound when $q$ is sufficiently large in terms of $n$.

\begin{theorem} \label{thm:main1}
For any fixed $n \geq 3$, 
\[f(n,q) \geq q^{q^{\iddots^{q - o(q)}}}
 \raisebox{-10pt}[0pt][0pt]{\hspace*{-24pt}\dbrace{11pt}}
 \raisebox{0pt}[0pt][0pt]{\hspace*{-11pt}\scriptsize{$n$-$1$}}
.\]
\end{theorem}

In particular, for $n = 3$, this says that $f(3,q) \geq q^{q - o(q)}$, a result we will prove by an appeal to the Lov\'asz local lemma. A key observation here is that it is not enough to apply the local lemma to the uniform random model where every word of a given length occurs with the same probability (though an approach of this form is discussed in~\cite{CR14}). Instead, we make use of a non-uniform random model which separates all instances of any given letter. 

For higher $n$, there are two different ways to proceed, one based on generalising the local lemma argument discussed above and another based on an explicit iterative construction which allows us to step up from the $Z_n$-case to the $Z_{n+1}$-case for all $n \geq 3$. This is in some ways analogous to the situation for hypergraph Ramsey numbers, where the Ramsey numbers of complete $3$-uniform hypergraphs determine the Ramsey numbers of complete $k$-uniform hypergraph for all $k \geq 4$. The difference here is that we are able to determine $f(3,q)$ very accurately, while the Ramsey number of the complete $3$-uniform hypergraph remains as elusive as ever (see~\cite{CFS15} for a thorough discussion).

This stepping-up method also allows us to address the weakness in Theorem~\ref{thm:main1}, that $n$ is taken to be fixed. Indeed, after suitable modification, the method proves sufficiently malleable that we can prove a tower-type lower bound even over a binary alphabet. This is the content of the next theorem, which is clearly tight up to an additive constant in the tower height.

\begin{theorem} \label{thm:main2}
\[f(n,2) \geq 2^{2^{\iddots^{2}}}
 \raisebox{-10pt}[0pt][0pt]{\hspace*{-24pt}\dbrace{11pt}}
 \raisebox{0pt}[0pt][0pt]{\hspace*{-11pt}\scriptsize{$n$-$4$}}
.\]
\end{theorem}

We also look more closely at the $n = 3$ case. This has been studied in some depth before, with Rytter and Shur~\cite{RS15} proving that $f(3,q) = O(2^q (q+1)!)$. We improve their result by a factor of roughly $q$ and show that this is tight up to a multiplicative constant.

\begin{theorem} \label{thm:f3q}
$f(3,q) = \Theta(2^q q!)$.
\end{theorem}

The paper is laid out as follows. For completeness, we will describe the simple proof of the upper bound on $f(n, q)$ in the next section. In Section~\ref{sec:local}, we will show how the local lemma can be used to prove Theorem~\ref{thm:main1}. We do this in two stages, first proving a lower bound for $f(3,q)$ which is sufficient for iteration and then addressing the general case. In Section~\ref{sec:stepup}, we discuss the stepping-up technique, first showing how to complete the second proof of Theorem~\ref{thm:main1} via this method and then how to modify the approach to give Theorem~\ref{thm:main2}. In Section~\ref{sec:z3}, we prove Theorem~\ref{thm:f3q}, determining $f(3,q)$ up to a constant factor. We conclude by discussing some further directions and open problems. Throughout the paper, we will use $\log$ to denote the logarithm base $2$. For the sake of clarity of presentation, we will also systematically omit floor and ceiling signs.

\section{The upper bound} \label{sec:upper}

The proof of the upper bound has two components. The first is the following simple lemma, due to Cooper and Rorabaugh~\cite{CR14}.

\begin{lemma} \label{lem:recur}
$f(n+1, q) \leq (f(n,q) + 1)(q^{f(n,q)} + 1) - 1$.
\end{lemma}

\begin{proof}
Consider a word of length $(f(n,q) + 1)(q^{f(n,q)} + 1) - 1$ of the form
\[\underbrace{\underbrace{ab \dots c}_{f(n,q)} x \underbrace{hi \dots k}_{f(n,q)} y \dots z \underbrace{rs \dots t}_{f(n,q)}}_{q^{f(n,q)} + 1}.\]
That is, we have $q^{f(n,q)} + 1$ words of length $f(n,q)$, each separated by an additional letter. By the definition of $f(n,q)$, each such word contains a copy of $Z_n$. Since there are $q^{f(n,q)} + 1$ such copies, two of them must be equal. As these two copies are separated by at least one letter, this yields a copy of $Z_{n+1}$.
\end{proof}

A naive application of Lemma~\ref{lem:recur} starting from $f(2, q) = 2q+1$ already yields a bound of the form
\[f(n,q) \leq q^{q^{\iddots^{2q + o(q)}}}
 \raisebox{-10pt}[0pt][0pt]{\hspace*{-24pt}\dbrace{11pt}}
 \raisebox{0pt}[0pt][0pt]{\hspace*{-11pt}\scriptsize{$n$-$1$}}
.\]
To improve the topmost exponent, we use the following refinement of Lemma~\ref{lem:recur}, due to Rytter and Shur~\cite{RS15}. The method works for all $n$, but for our purposes it will suffice to consider the case  $n = 3$.

\begin{lemma} \label{lem:base}
$f(3,q) \leq 2^{q+1} (q+1)!$.
\end{lemma}

\begin{proof}
Say that a word $w$ is $2$-minimal if it contains $Z_2$ but every subword avoids $Z_2$. If $w$ is $2$-minimal, it is easy to check that either $w = aaa$ for a fixed letter $a$ or $w = ab_1^{j_1} \dots b_r^{j_r} a$, where all of the $b_i$ are distinct and $j_i \in \{1, 2\}$ for all $i$. Thus, the number $t(2, q)$ of $2$-minimal words over an alphabet of size $q$ is 
\[q + \sum_{r=1}^{q-1} q(q-1)\dots (q - r) 2^r \leq q! \sum_{r = 0}^{q-1} 2^r \leq 2^q q! - 1.\]
Now consider a word of length $(f(2,q) + 1)(t(2,q) + 1) - 1$ of the form
\[\underbrace{\underbrace{ab \dots c}_{f(2,q)} x \underbrace{hi \dots k}_{f(2,q)} y \dots z \underbrace{rs \dots t}_{f(2,q)}}_{t(2,q) + 1}.\]
Each word of length $f(2,q)$ contains a $2$-minimal word. Therefore, since there are $t(2,q) + 1$ words of length $f(2,q)$ and only $t(2,q)$ $2$-minimal words, two of the corresponding $2$-minimal words must be the same. This easily yields a copy of $Z_3$. Since
\[(f(2,q) + 1)(t(2,q) + 1) - 1 \leq (2q + 2) 2^q q! = 2^{q+1} (q+1)!,\]
the result follows.
\end{proof}

The interested reader may wish to skip to Section~\ref{sec:z3}, where we improve the estimate above to $f(3,q) = O(2^q q!)$ and show that this is tight up to a constant factor. For now, we continue to focus on the general case, combining Lemmas~\ref{lem:recur} and \ref{lem:base} to prove the required upper bound on $f(n,q)$.

\begin{theorem} For $n \geq 3$ and $q \geq 35$,
\[f(n,q) \leq q^{q^{\iddots^{q}}}
 \raisebox{-10pt}[0pt][0pt]{\hspace*{-24pt}\dbrace{11pt}}
 \raisebox{0pt}[0pt][0pt]{\hspace*{-11pt}\scriptsize{$n$-$1$}}
.\]
\end{theorem}

\begin{proof}
We will prove by induction on $n$ the stronger result that 
\[q f(n, q) \leq  q^{q^{\iddots^{q}}}
 \raisebox{-10pt}[0pt][0pt]{\hspace*{-24pt}\dbrace{11pt}}
 \raisebox{0pt}[0pt][0pt]{\hspace*{-11pt}\scriptsize{$n$-$1$}}
.\]
For the base case $n = 3$, the result follows from Lemma~\ref{lem:base} since $f(3,q) \leq 2^{q+1} (q+1)! \leq q^{q-1}$ for $q \geq 35$. Writing $f := f(n,q)$, we will assume that $q f \leq T$, for some $T \geq q^{q}$, and show that $f(n+1, q) \leq q^{T - 1}$, from which the required result follows. By Lemma~\ref{lem:recur}, we have
\[f(n+1, q) \leq (f+1)(q^f+1) - 1 = f q^f + q^f + f \leq (f+2)q^f\]
and, therefore,
\[f(n+1, q) \leq \left(\frac{T}{q} + 2\right) q^{T/q} \leq T q^{T/q} \leq q^{T-1},\]
as required.
\end{proof}

\section{Applying the local lemma} \label{sec:local}

As an illustration of the main idea behind our proof, we will initially focus on the case $n = 3$, showing that $f(3,q) \geq q^{q- o(q)}$. In order to state the version of the 
Lov\'asz local lemma that we will need (see, for example, \cite{AS}), we say that a directed graph $D = (V, E)$ with $V=\{1,\ldots,n\}$ is a \emph{dependency digraph} for the set of events $A_1, A_2, \dots, A_n$ if for each $i$, $1 \leq i \leq n$, the event $A_i$ is mutually independent of all the events $\{A_j:(i,j) \not \in E\}$.

\begin{lemma} \label{lem:LLL}
Suppose that $D=(V,E)$ is a dependency digraph for the events $A_1, A_2, \dots, A_n$ with all outdegrees at most $d$. If $\textrm{Pr}[A_i] \leq p$ for all $i$ and $e p (d+1) \leq 1$, then 
$$\textrm{Pr}\left[ \bigcap_{i=1}^n \overline{A_i} \right] \geq \left(1-\frac{1}{d+1}\right)^n \geq e^{-n/d} > 0.$$
\end{lemma}

\begin{theorem} \label{thm:3LLL}
$f(3,q) \geq q^{q- o(q)}$.
\end{theorem}

\begin{proof}
We begin by splitting our alphabet arbitrarily into $t = \log q$ parts $L_1, L_2, \dots, L_t$, each of size $S:=\frac{q}{\log q}$. We generate a random word by placing letters in a series of successive intervals $I_1, I_2, \dots$, each of length $S$, as follows: first, fill $I_1$ with a random permutation of the letters from $L_1$; then apply the same process in $I_j$ for each $j = 2, 3, \dots, t$, that is, fill $I_j$ with a permutation of the letters from $L_j$; for interval $I_{t+1}$ we reuse the letters from $L_1$, for interval $I_{t+2}$ we reuse the letters from $L_2$ and so on, where for the interval $I_{it + j}$ we reuse the letters from $L_j$.

Note that, because of how we place the letters, for any two instances of the same letter, there are at least $t-1$ consecutive intervals $I_j$ of length $S$ between them. That is, every copy of $Z_2$ has length at least $T = (t-1)S$ and includes $t-1$ consecutive intervals $I_j$. Therefore, in order to find a copy of $Z_3$ in a word of this form, we must find two disjoint equal intervals of length $T$ consisting of $t-1$ intervals, each with the same $t-1$ permutations of length $S$. We will now use the local lemma to show that there is a word of length $N \geq q^{q - o(q)}$ containing no such pair and, thus, containing no copy of $Z_3$.

Suppose, therefore, that we have used the process described above to generate a random word of length $N=S!^{t-1}=q^{q-o(q)}$. 
Let $A_1, A_2, \dots$ be the collection of events corresponding to the existence of two disjoint intervals of length $T$, each consisting of $t-1$ of the intervals of length $S$ described above, containing the same subword. Note that any such pair of intervals of length $T$ will overlap with at most $4tN/S$ other such pairs of intervals of length $T$. Indeed, there are at most 
$2(2t-3)$ ways to choose an interval of length $T$ overlapping with one of the two intervals forming the pair. For the other interval 
there are at most $N/S$ possibilities, each given by the first interval $I_j$ of length $S$ it contains.  

Note that $\textrm{Pr}[A_i]=S!^{-(t-1)}$ for each $i$. Applying the local lemma, Lemma~\ref{lem:LLL}, with $p = S!^{-(t-1)}$ and $d = 4tN/S$, we see that since
$ep(d+1) \leq 12t/S < 1$, there exists a word of length $N$ such that none of the events $A_i$ hold, as required. By the discussion above, this word contains no copy of $Z_3$, so the proof is complete.
\end{proof}

We also note a slight strengthening of this result which will be useful in the next section. In the proof, we will freely use notation from the proof above.

\begin{theorem} \label{thm:3LLL2}
There are at least $q^{q^{q-o(q)}}$ words $w$ of length $q^{q - o(q)}$ over an alphabet of size $q$ such that $w$ avoids $Z_3$ and there is a distinguished letter $d$ such that any subword of $w$ not containing the letter $d$ avoids $Z_2$.
\end{theorem}

\begin{proof}
The proof is almost exactly the same as the proof of Theorem~\ref{thm:3LLL}, except we set aside the distinguished letter $d$ at the start, only using it immediately after each interval of the form $I_{it}$ to separate it from the interval $I_{it+1}$. By construction, the word between any two successive instances of $d$ will consist of the intervals $I_{it+1}, I_{it+2}, \dots, I_{(i+1)t}$. But the union of these intervals contains no repeated letters and, hence, no copy of $Z_2$, as required.

To count the number of words, note that the number of possible $N$-letter words generated by our random process is equal to $(S!)^{N/S}=q^{(1 - o(1))N}$, each occurring with the same probability. Since there are fewer than $(N/S)^2$ bad events $A_1, A_2, \dots$, each of which is independent of all but $4tN/S$ of the others, the local lemma, Lemma~\ref{lem:LLL}, tells us that with probability at least $e^{-(N/S)^2/(4tN/S)} = e^{-N/4tS}$ none of these bad events happen, so the process generates an appropriate word. In fact, there must be at least
\[ e^{-N/4tS}\cdot q^{(1 -o(1))N} \geq q^{(1 - o(1))N} \geq q^{q^{q - o(q)}}\]
appropriate words, completing the proof.
\end{proof}

The remainder of this section will be concerned with generalising the proof of Theorem~\ref{thm:3LLL} to give a local lemma proof of Theorem~\ref{thm:main1}. The reader who is willing to accept our word that such a generalisation is possible may skip to the start of the next section to see how a recursive procedure may also be used to finish the job. For the resolute, we state a more general form of the Lov\'asz local lemma (see \cite{AS}).

\begin{lemma} \label{lem:LLL2}
Suppose that $D=(V,E)$ is a dependency digraph for the events $A_1, A_2, \dots, A_n$. If there are real numbers $x_1,\ldots,x_n$ such that $0 \leq x_i < 1$ and $\textrm{Pr}[A_i] \leq x_i \prod_{(i,j) \in E} (1-x_j)$ for all $1 \leq i \leq n$, then 
$$\textrm{Pr}\left[ \bigcap_{i=1}^n \overline{A_i} \right] \geq \prod_{i=1}^n \left(1-x_i\right) > 0.$$
\end{lemma}

{\bf First proof of Theorem~\ref{thm:main1}:}
We will generate random words in the same manner as in the proof of Theorem~\ref{thm:3LLL}. That is, we split our alphabet arbitrarily into $t = \log q$ parts $L_1, L_2, \dots, L_t$, each of size $S := \frac{q}{\log q}$, and generate a random word by placing letters in a series of successive intervals $I_1, I_2, \dots$, each of length $S$, as follows: first, fill $I_1$ with a random permutation of the letters from $L_1$; then apply the same process in $I_j$ for each $j = 2, 3, \dots, t$, that is, fill $I_j$ with a permutation of the letters from $L_j$; for interval $I_{t+1}$ we reuse the letters from $L_1$, for interval $I_{t+2}$ we reuse the letters from $L_2$ and so on, where for the interval $I_{it + j}$ we reuse the letters from $L_j$. Once again, we note that any two instances of the same letter must be at least a distance $T := (t-1)S$ apart, and so the shortest copy of $Z_2$ has length at least $T$.

Define $y_1 = T$ and, for $1 \leq i \leq n-2$, $y_{i+1}=(q/\log^4q)^{y_i}$. For each $1 \leq i \leq n-2$, we consider all bad events $A_{i,1}, \dots, A_{i, r_i}$ corresponding to the existence of two disjoint identical intervals of length $y_i$ appearing at distance at most $y_{i+1}$ from one another. If none of these events occur in a word $w$ generated as described above, we see, since every copy of $Z_2$ in $w$ has length at least $T$ and any two identical intervals of length $T$ are at least $y_2$ apart, that every copy of $Z_3$ in $w$ has length at least $y_2$. In turn, since any two identical intervals of length $y_2$ are at least $y_3$ apart, this implies that every copy of $Z_4$ in $w$ has length at least $y_3$. Iterating, we see that every copy of $Z_n$ in $w$ must have length at least $y_{n-1}$. Hence, for $w$ to contain $Z_n$, it must have length at least $y_{n-1}$, which is easily seen to satisfy the inequality
\[y_{n-1} \geq q^{q^{\iddots^{q-o(q)}}}
 \raisebox{-10pt}[0pt][0pt]{\hspace*{-24pt}\dbrace{11pt}}
 \raisebox{0pt}[0pt][0pt]{\hspace*{-11pt}\scriptsize{$n$-$1$}}.\]
It therefore remains to show that there exists an appropriate $w$ of length $y_{n-1} - 1$ such that none of the bad events $A_{i,1}, \dots, A_{i, r_i}$ for $i = 1, 2, \dots, n-2$ occur. To apply the local lemma, we need to analyse the dependencies between different events. Suppose, therefore, that $i$ and $j$ are fixed and we wish to determine how many of the events $A_{j, 1}, \dots, A_{j, r_j}$ a particular $A_{i,k}$ depends on. 

For $i \leq j$, there are at most $8 y_j y_{j+1}$ events $A_{j, \ell}$ that depend on $A_{i,k}$. Indeed, one of the elements in the pair of intervals corresponding to $A_{j, \ell}$ must be equal to one of the endpoints from the pair of intervals corresponding to $A_{i, k}$. There are $4$ choices for the endpoint and $2y_j$ choices for which of the elements corresponds to this endpoint.  Once these choices are made, they fix one of the intervals in the pair  corresponding to $A_{j, \ell}$ and the other interval may be chosen arbitrarily within distance $y_{j+1}$ from the first one, so there are $y_{j+1}$ choices. A similar argument applies when $i > j$ to show that there are at most $8 y_iy_{j+1}$ events $A_{j, \ell}$ that depend on $A_{i,k}$.

To estimate $\textrm{Pr}[A_{i,k}]$, note that any interval of length $y_i$ will fully contain at least $y_i/S - 2$ successive intervals of the form $I_j$ and, therefore, 
$$\textrm{Pr}[A_{i,k}] \leq S!^{-y_i/S + 2} \leq \left(\frac{e}{S}\right)^{y_i - 2S} \leq \left(\frac{\log^2 q}{q}\right)^{y_i}.$$ 
We will now apply the local lemma with $x_i := x_{i,k} = (\log^3 q/q)^{y_i}$ for all events $A_{i,k}$. By using that $n$ is fixed together with the inequality $1 - x \geq e^{-2x}$ for $0 \leq x \leq \frac{1}{2}$, we see that
\[(1 - x_j)^{8 y_{j+1}} = \left(1 - \left(\frac{\log^3 q}{q}\right)^{y_j}\right)^{8 y_{j+1}} \geq e^{-16 \left(\frac{\log^3 q}{q}\right)^{y_j} y_{j+1}} = e^{-16 (\log q)^{-y_j}}  \geq 2^{-1/n y_j}\]
and, therefore,
\begin{align*}
\textrm{Pr}[A_{i,k}] & \leq \left(\frac{\log^2 q}{q}\right)^{y_i} = \left(\frac{\log^3 q}{q}\right)^{y_i} \left(\frac{1}{\log q}\right)^{y_i}\\
& \leq x_i 2^{-y_i} = x_i 2^{-y_i/n} \cdot 2^{-y_i/n} \cdots 2^{-y_i/n} \\
& \leq x_i \prod_{j = 1}^{i-1} (1 - x_j)^{8 y_i y_{j+1}} \prod_{j = i}^{n-2} (1 - x_j)^{8 y_j y_{j+1}}\\
& \leq x_i \prod_{j=1}^{n-2} \prod_{(j,\ell) \sim (i,k)} (1-x_j).
\end{align*}
We may therefore apply the local lemma to obtain the desired word, completing the proof.
{\hfill$\Box$\medskip}

\section{Stepping up} \label{sec:stepup}

We will begin this section by completing our second proof of Theorem~\ref{thm:main1}. This is based on a simple recursion encapsulated in Lemma~\ref{lem:stepup} below. To state this result, we need a few definitions.

Let $m(n,q)$ denote the number of words over an alphabet of size $q$ which avoid $Z_n$. Note the inequality
$f(n,q) \geq \log_q m(n,q)$, which follows since the number of words over a $q$-letter alphabet of length less than $f$ is $1+q+q^2+\cdots+q^{f-1} \leq q^f$.
Let $S(n,q)$ denote the set of all words $w$ over an alphabet of size $q$ which avoid $Z_n$ and have a distinguished letter, say $d$, such that any subword of $w$ not containing the letter $d$ avoids $Z_{n-1}$. We let $F(n,q)$ denote the length of the longest word in $S(n,q)$ and $M(n,q) = |S(n,q)|$. By definition,
$f(n,q) > F(n,q)$ and $m(n,q) \geq M(n,q)$. 

\begin{lemma} \label{lem:stepup}
$$M(n+1,q+2) \geq M(n,q)!$$ and $$F(n+1,q+2) \geq M(n,q).$$
\end{lemma}
\begin{proof}
Let $c$ denote the distinguished letter in the words in $S(n,q)$. Writing $M:=M(n,q)$, consider any one of the $M!$
orderings of the words in $S(n,q)$, say $w_1,w_2,\ldots,w_M$. For $i$ odd, let $u_i$ be obtained from $w_i$ by changing every $c$ to $c_1$. For $i$ even, let $u_i$ be obtained from $w_i$ by changing every $c$ to $c_0$. Add a new distinguished letter $d$ and consider the word $w=u_1du_2du_3du_4d\ldots d u_M$ formed by placing a $d$ between each $u_i$ and $u_{i+1}$ and concatenating the sequence. The number of letters in $w$ is $q+2$, consisting of the original $q-1$ nondistinguished letters, the new letters $c_0,c_1$ replacing the old letter $c$ and the new distinguished letter $d$. The number of possible choices for $w$ is $M!$, one for each ordering of the words in $S(n,q)$. Moreover, the length of $w$ is $\sum_{w \in S(n,q)} (|w|+1) - 1 \geq |S(n,q)|=M(n,q)$. It will therefore suffice to show that $w \in S(n+1,q+2)$. 

Note that any subword of $w$ which does not contain the distinguished letter $d$ is a subword of some $u_i$ and, since $u_i$ is a copy of a word in $S(n,q)$, it does not contain $Z_n$. It only remains to show that $w$ does not contain a copy of $Z_{n+1}$. Suppose for contradiction that it does and let this subword be $XYX$, with $X$ a copy of $Z_{n}$. Neither $X$ contains two or more copies of the letter $d$, since between any two consecutive copies of $d$ there is a unique word which cannot then appear in both copies of $X$. If $X$ contains no $d$, then each of the two copies of $X$ is a subword of a $u_i$ (not necessarily the same). However, no $u_i$ contains $Z_{n}$, contradicting the fact that $X$ is a copy of $Z_{n}$. So each $X$ contains exactly one $d$. Write $X=ABA$ with $A$ a copy of $Z_{n-1}$. As $X$ contains exactly one $d$, this copy of $d$ must be in $B$ and each copy of $X$ is entirely contained in a subword of $w$ of the form $u_idu_{i+1}$ for some $i$ (which will be a  different $i$ for the left and right copy of $X$). As $i$ and $i+1$ have different parity, the distinguished letter of $u_i$ is not in $u_{i+1}$ and the distinguished letter of $u_{i+1}$ is not in $u_i$. Thus, the left and right copies of $A$ do not contain the distinguished letters of $u_i$ or of $u_{i+1}$. However, $w_i$ and $w_{i+1}$ are both in $S(n,q)$, so these copies of $A$ cannot contain a copy of $Z_{n-1}$, contradicting the fact that $A$ is a copy of $Z_{n-1}$. 
\end{proof} 

We may now complete our second proof of Theorem~\ref{thm:main1}.
\vspace{2mm}

{\bf Second proof of Theorem~\ref{thm:main1}:} We will begin by proving inductively that
\[M(n,q) \geq q^{q^{\iddots^{q - o(q)}}}
 \raisebox{-10pt}[0pt][0pt]{\hspace*{-24pt}\dbrace{11pt}}
 \raisebox{0pt}[0pt][0pt]{\hspace*{-11pt}\scriptsize{$n$}}\]
 for all $n \geq 3$. For $n=3$, this follows from Theorem~\ref{thm:3LLL2}. For the induction step, we use Lemma~\ref{lem:stepup} to conclude that
\[M(n+1, q+2) \geq M(n, q)! \geq \left(\frac{M(n,q)}{e}\right)^{M(n,q)} \geq q^{M(n,q)},\]
which easily implies the required result. To complete the proof of the theorem, note that Theorem~\ref{thm:3LLL} handles the case $n=3$, while, for $n \geq 4$, Lemma~\ref{lem:stepup} and our bound on $M(n,q)$ together imply that
\[f(n,q) \geq M(n-1, q - 2) \geq q^{q^{\iddots^{q - o(q)}}}
 \raisebox{-10pt}[0pt][0pt]{\hspace*{-24pt}\dbrace{11pt}}
 \raisebox{0pt}[0pt][0pt]{\hspace*{-11pt}\scriptsize{$n$-$1$}},\]
 as required.
{\hfill$\Box$\medskip}

We now turn to the proof of Theorem~\ref{thm:main2}. This is similar in broad outline to the proof of Theorem~\ref{thm:main1} above, where we produced words which are $Z_{n+1}$-free by concatenating a collection of $Z_n$-free words, separating them by instances of an extra distinguished letter. However, here, in order to avoid adding extra letters to our alphabet, we will instead separate our $Z_n$-free words with long strings of $1$s. This alteration makes the proof considerably more delicate. 

To proceed, we let $1_{x}$ denote the word consisting of $x$ ones and $B(n)$ the largest set of binary words $w$ of the same length with the following properties: 

\begin{enumerate} 
\item $w$ begins and ends with a zero. 
\item $w$ does not contain $1_{2n+1}$ as a subword. 
\item Any subword of $1_{2n}w1_{2n}$ not containing $1_{2n}$ is $Z_{n-1}$-free. 
\item $1_{2n}w1_{2n}$ is $Z_n$-free. 
\item Let $w'$ be obtained from $w$ by adding a one to each copy of $1_{2n}$ in $w$. Then $1_{2n+1}w'1_{2n+1}$ is $Z_n$-free.
\end{enumerate}

The key to proving Theorem~\ref{thm:main2} is the following lemma relating $|B(n+1)|$ to $|B(n)|$.

\begin{lemma} \label{lem:stepup2}
For $n \geq 6$, $$|B(n+1)| \geq |B(n)|!$$
\end{lemma}
\begin{proof} 
Let $w_1,\ldots,w_b$ be a permutation of the words in $B(n)$. Let $u_i=w_i$ if $i$ is odd and otherwise $u_i$ is obtained from $w_i$ by adding a one to each copy of $1_{2n}$ in $w_i$. The proof will follow similar lines to the proof of Lemma~\ref{lem:stepup}, but with the word $1_{2n+2}$ serving as the analogue of a distinguished letter. That is, instead of introducing new letters, we use a special subword consisting only of ones.

To that end, let $w$ be the word $u_11_{2n+2}u_21_{2n+2}u_31_{2n+2}\ldots 1_{2n+2}u_b$ formed by placing a copy of $1_{2n+2}$ between each $u_i$ and $u_{i+1}$ and concatenating the sequence. To prove the lemma, it will suffice to show that $w$ satisfies the five properties required for a word to be in $B(n+1)$. As each $w_i \in B(n)$, each $u_i$ begins and ends with a zero, and so $w$ also begins and ends with a zero, verifying the first property. As every subword of $u_i$ consisting only of ones has length at most $2n+1$ and, since each $u_i$ begins and ends with a zero, there is a zero before and after each $1_{2n+2}$ occurrence, $w$ does not contain $1_{2n+3}$, verifying the second property.

The third property asks that any subword of $1_{2n+2}w1_{2n+2}$ not containing $1_{2n+2}$ is $Z_n$-free. Any such subword must be contained in $1_{2n+2}u_i1_{2n+2}$ for some $i$ but not containing the first or last letter. 
Recall also that $u_i$ starts and ends with $0$. By using the fourth and fifth properties of $B(n)$, we see that the word $1_{2n}u_i1_{2n}$ is $Z_n$-free when $i$ is odd and the word $1_{2n+1}u_i1_{2n+1}$ is $Z_n$-free when $i$ is even. Therefore, any copy $Z$ of $Z_n$ must start at the second letter of $1_{2n+2}u_i1_{2n+2}$ or end at the second to last letter of $1_{2n+2}u_i1_{2n+2}$ for some odd $i$. Write $Z=ABACABA$ with $A$ a copy of $Z_{n-2}$. As there are four copies of $A$ but at most two possible copies of $1_{2n+1}$ in $Z$, and $Z$ begins at the second letter or ends at the second to last letter of $1_{2n+2}u_i1_{2n+2}$, $A$ must be all ones and have length at most $2n$. However, $A$ is a copy of $Z_{n-2}$ and hence has length at least $2^{n-2}-1>2n$ (since $n \geq 6$), a contradiction. This verifies the third property. 

We next verify the fourth property, that $1_{2n+2}w1_{2n+2}$ is $Z_{n+1}$-free. Suppose, for contradiction, that $1_{2n+2}w1_{2n+2}$ contains a copy $Z'$ of $Z_{n+1}$ and this copy is of the form $XYX$, where $X$ is a copy of $Z_n$. Neither copy of $X$ contains one of the subwords $u_i$ used to make $w$, as otherwise the other copy of $X$ would have to contain an identical subword $u_{i'}$, but $u_i$ and $u_{i'}$ are distinct. Hence, the left copy of $X$ must be in $1_{2n+2}u_{1}$ or $u_{i}1_{2n+2}u_{i+1}$ for some $i$ and the right copy of $X$ must be in $u_{b}1_{2n+2}$ or $u_{i'}1_{2n+2}u_{i'+1}$ for some $i'$. Write $X=DED$ with $D$ a copy of $Z_{n-1}$. We will assume, without loss of generality, that the left copy of $X$ is in $u_{i}1_{2n+2}u_{i+1}$ with $i$ odd (the other case may be handled similarly).

We first show that $X$ contains the copy of $1_{2n+2}$. If $X$ is in $u_i1_{2n}$ or in $1_{2n+1}u_{i+1}$, then the fact that $X$ is a copy of $Z_n$ would contradict the fourth and fifth properties of $B(n)$, respectively. If $X$ is in $u_i1_{2n+1}$ and contains the last letter of $1_{2n+1}$, then either the right copy of $D$ is a subword of $1_{2n+1}$, contradicting the fact that $D$ is a copy of $Z_{n-1}$ (which must have length at least $2^{n-1}-1>2n$), or the right copy of $D$ contains $1_{2n+1}$, forcing the left copy of $D$ to also contain $1_{2n+1}$ but be a subword of $u_{i}$, which is $1_{2n+1}$-free. In any case, we see that $X$ contains the copy of $1_{2n+2}$.

If $E$ does not intersect the copy of $1_{2n+2}$, then $1_{2n+2}$ is entirely contained in one of the copies of $D$ and hence also in the other copy of $D$, which is entirely in $u_i$ or $u_{i+1}$, a contradiction. Hence, $E$ intersects the copy of $1_{2n+2}$. Thus, the left copy of $D$ is in $u_i1_{2n+1}$ and the right copy of $D$ is in $1_{2n+1}u_{i+1}$. We now split into cases.

Case 1: $D$ contains $1_{2n+1}$ as a subword.

In this case, as $u_i$ does not contain a copy of $1_{2n+1}$ and ends with $0$, the left copy of $D$ is in $u_i1_{2n+1}$ and ends at the last letter. Writing $D=ABA$ with $A$ a copy of $Z_{n-2}$, we see that the right copy of $A$ contains at most $2n$ letters, as otherwise the left copy of $A$, which is a subword of $u_i$, would contain $1_{2n+1}$. But $Z_{n-2}$ has length at least $2^{n-2}-1>2n$, contradicting the fact that $A$ is a copy of $Z_{n-2}$. 

Case $2$: $D$ contains $1_{2n}$ as a subword but does not contain $1_{2n+1}$. 

In this case, by the construction of $u_{i+1}$, the right copy of $D$ must be a subword of $1_{2n}u1_{2n}$ with $u$ a subword of $w_{i+1}$ which begins and ends with a $0$ and is $1_{2n}$-free. Writing $D=ABA$ with $A$ a copy of $Z_{n-2}$, we see, since $2^{n-2}-1 > 2n$, that $A$ contains $1_{2n}$ as a strict subword. But if, for example, $A$ contains $1_{2n}0$, this easily contradicts the fact that $D$ is a subword of $1_{2n}u1_{2n}$ with at most two copies of $1_{2n}$.

Case $3$: $D$ does not contain $1_{2n}$ as a subword. 

In this case, considering the left copy of $D$, by the third property of $B(n)$, $D$ is $Z_{n-1}$-free, contradicting that $D$ is a copy of $Z_{n-1}$. 
This completes the verification of the fourth property of $B(n+1)$. 

Finally, we need to verify the fifth property. This says that if $w'$ is obtained from $w$ by adding a one to each copy of $1_{2n+2}$ in $w$, then $1_{2n+3}w'1_{2n+3}$ is $Z_{n+1}$-free. The proof of this is almost identical to the proof of the fourth property, but we include it for completeness. 

Suppose, for contradiction, that there is a copy $Z'$ of $Z_{n+1}$ in $1_{2n+3}w'1_{2n+3}$ and this copy is of the form $XYX$, where $X$ is a copy of $Z_n$. 
Note that while creating $w'$ we did not change any of the words $u_i$, since they start and end with $0$ and contain no copy of $1_{2n+2}$ by the second property of  $B(n)$.
Neither copy of $X$ contains a $u_i$ used to make $w$, as otherwise the other copy of $X$ would have to contain an identical subword $u_{i'}$, but $u_i$ and $u_{i'}$ are distinct. Hence, the left copy of $X$ must be in $1_{2n+3}u_{1}$ or $u_{i}1_{2n+3}u_{i+1}$ for some $i$ and the right copy of $X$ must be in $u_{b}1_{2n+3}$ or $u_{i'}1_{2n+3}u_{i'+1}$ for some $i'$. Write $X=DED$ with $D$ a copy of $Z_{n-1}$. We will assume, without loss of generality, that the left copy of $X$ is in $u_{i}1_{2n+3}u_{i+1}$ with $i$ odd (the other case may be handled similarly).

We first show that $X$ contains the copy of $1_{2n+3}$. If $X$ is a subword of $u_i1_{2n}$, then the fact that $X$ is a copy of $Z_n$ contradicts the fourth  property of $B(n)$. If $X$ is a subword of $1_{2n+1}u_{i+1}$,  then the fact that $X$ is a copy of $Z_n$ contradicts the fifth property of $B(n)$. If $X$ is in $u_i1_{2n+2}$ and contains one of the last two letters, then either the right copy of $D$ is a subword of $1_{2n+2}$, contradicting the fact that $D$ is a copy of $Z_{n-1}$ (which must have length at least  $2^{n-1}-1>2n+2$), or the right copy of $D$ contains $1_{2n+1}$, forcing the left copy of $D$ to also contain $1_{2n+1}$ but be a subword of $u_{i}$, which is $1_{2n+1}$-free. If $X$ is in $1_{2n+2}u_{i+1}$ and contains the first letter of $1_{2n+2}$, then either the left copy of $D$ is a subword of $1_{2n+2}$, again a contradiction, or the left copy of $D$ contains $1_{2n+2}$, forcing the right copy of $D$ to also contain $1_{2n+2}$ but be a subword of $u_{i+1}$, which is $1_{2n+2}$-free. In any case, we see that $X$ contains the copy of $1_{2n+3}$.

If $E$ does not intersect the copy of $1_{2n+3}$, then $1_{2n+3}$ is entirely contained in one of the copies of $D$ and hence also in the the other copy of $D$, which is entirely in $u_i$ or $u_{i+1}$, a contradiction. Hence, $E$ intersects the copy of $1_{2n+3}$. Thus, the left copy of $D$ is in $u_i1_{2n+2}$ and the right copy of $D$ is in $1_{2n+2}u_{i+1}$. We again split into cases.

Case 1: $D$ contains $1_{2n+1}$ as a subword.

In this case, as $u_i$ does not contain a copy of $1_{2n+1}$, the left copy of $D$ is in $u_i1_{2n+2}$ and ends at one of the last two letters. Writing $D=ABA$ with $A$ a copy of $Z_{n-2}$, we see that the right copy of $A$ contains at most $2n$ letters, as otherwise the left copy of $A$, which is a subword of $u_i$, would contain $1_{2n+1}$. But $Z_{n-2}$ has length at least $2^{n-2}-1>2n$, contradicting the fact that $A$ is a copy of $Z_{n-2}$. 

Case $2$: $D$ contains $1_{2n}$ as a subword but does not contain $1_{2n+1}$. 

In this case, by the construction of $u_{i+1}$, the right copy of $D$ must be a subword of $1_{2n}u1_{2n}$ with $u$ a subword of $w_{i+1}$ which begins and ends with a $0$ and is $1_{2n}$-free. Writing $D=ABA$ with $A$ a copy of $Z_{n-2}$, we see, since $2^{n-2}-1 > 2n$, that $A$ contains $1_{2n}$ as a strict subword. But if, for example, $A$ contains $0 1_{2n}$, this easily contradicts the fact that $D$ is a subword of $1_{2n}u1_{2n}$ with at most two copies of $1_{2n}$.

Case $3$: $D$ does not contain $1_{2n}$ as a subword. 

In this case, considering the left copy of $D$, by the third property of $B(n)$, $D$ is $Z_{n-1}$-free, contradicting that $D$ is a copy of $Z_{n-1}$. 
We have therefore verified the fifth property of $B(n+1)$, completing the proof of the lemma. 
\end{proof}

We round off the section by proving Theorem~\ref{thm:main2}, which states that there are binary words avoiding $Z_n$ of length at least a tower of twos of height $n-4$, that is,
\[f(n,2) \geq 2^{2^{\iddots^{2}}}
 \raisebox{-10pt}[0pt][0pt]{\hspace*{-24pt}\dbrace{11pt}}
 \raisebox{0pt}[0pt][0pt]{\hspace*{-11pt}\scriptsize{$n$-$4$}}
.\]

{\bf Proof of Theorem~\ref{thm:main2}:} We will begin by proving inductively that 
\[|B(n)| \geq 2^{2^{\iddots^{2}}}
 \raisebox{-10pt}[0pt][0pt]{\hspace*{-24pt}\dbrace{11pt}}
 \raisebox{0pt}[0pt][0pt]{\hspace*{-11pt}\scriptsize{$n$-$3$}}
\]
for all $n \geq 6$. For the base case, note that $|B(6)| \geq 2^{4} = 2^{2^{2}}$ as all binary words of length $6$ beginning and ending with a $0$ have the five desired properties. For the induction step, we use Lemma~\ref{lem:stepup2} to conclude that
\[|B(n+1)| \geq |B(n)|! \geq 2^{|B(n)|}\]
for $|B(n)| \geq 4$, which easily gives the required result. To complete the proof of the theorem, we simply note that since all the words in $B(n)$ have the same length, their common length must be at least $\log_2 |B(n)|$.
{\hfill$\Box$\medskip}

\section{Determining $f(3,q)$ up to a constant factor} 
\label{sec:z3}

In this section, we prove Theorem~\ref{thm:f3q}, which determines the value of $f(3,q)$ up to an absolute constant. We begin by proving the upper bound. In the proof, we will say that an interval is {\it constant} if only one letter appears in that interval.

\begin{theorem}
For $q > 3$, $f(3,q) <  3e^{1/2} 2^q q!$
\end{theorem}

\begin{proof}
Let $w=a_1a_2\ldots a_n$ be a word of length $n$ over a $q$-letter alphabet which does not contain $Z_3$. Observe that there are no two intervals of length three in $w$ which are disjoint, non-consecutive and constant with respect to a given letter, as otherwise $w$ contains $Z_3$. For each letter in our alphabet, if there is a constant interval of length three in that letter, we delete this interval and one of the letters immediately before or after it so that no constant intervals of length three in that letter remain. This process deletes at most $q$ intervals of length at most four, leaving $q+1$ disjoint intervals of $w$ of total length at least $n-4q$ with the property that each such interval $J$ has no constant word of length three. 

If such an interval has two consecutive letters that are identical, replace it by a single instance of the same letter to obtain a new word $w'$ on a reduced interval $J'$. The word $w'$ has no two consecutive identical letters and $|J'| \geq |J|/2$. By the pigeonhole principle, each interval of length $q+1$ in $J'$ contains a copy of $Z_2$, and hence a minimal copy of $Z_2$. Each minimal copy of $Z_2$ in $J'$ consists of an interval $[i,j]$ with $i+2 \leq j \leq i+q$, where, for $i \leq k < l \leq j$, we have $a_k=a_l$ if and only if $k=i$ and $l=j$. The length of such a minimal copy of $Z_2$ is $j-i+1$ and it contains $j-i$ distinct letters. The number of intervals of length $q+1$ in $J'$ is $\max(|J'|-q,0)$. Let $m$ be the total number of such intervals of length $q+1$ taken over all of the at most $q+1$ intervals $J'$. Then 
$$m \geq \frac{n-4q}{2}-q\cdot (q+1) = \frac{n}{2} - q^2-3q.$$ 

Note that each minimal copy of $Z_2$ in $J'$ comes from a minimal copy of $Z_2$ in $J$, with each internal letter $x$ either originally coming from $x$ or $xx$. 
Thus, each minimal copy of $Z_2$ of length $s$ in some $J'$ comes from one of $2^{s-2}$ possible minimal copies of $Z_2$ in $w$. Note also that in the intervals of $w$, we cannot have three copies of the same minimal $Z_2$, as otherwise the first and last copy of $Z_2$ would be disjoint and separated by at least one letter, giving rise to a copy of $Z_3$. By the pigeonhole principle, if we get the same minimal copy of $Z_2$ in the reduced intervals more than $2^{s-1}$ times, then we get three identical minimal copies of $Z_2$ in the original word $w$, giving a copy of $Z_3$, a contradiction. Let $r_s$ be the number of minimal copies of $Z_2$ of length $s$ we get in total across the reduced intervals. As there are $q!/(q-s+1)!$ possible minimal copies of $Z_2$ of length $s$ with no two consecutive letters equal, we have $r_s \leq 2^{s-1} q!/(q-s+1)!$. Also, each copy of $Z_2$ of length $s$ is in at most $q+2-s$ intervals of length $q+1$. Hence, 
\begin{eqnarray*} 
m & \leq & \sum_{s=3}^{q+1} (q+2-s)r_s \leq  \sum_{s=3}^{q+1} (q+2-s)2^{s-1} q!/(q-s+1)!=\sum_{t=0}^{q-2} (t+1)2^{-t}t!^{-1}2^{q}q! \\ 
& = & 2^{q}q!\sum_{t=0}^{\infty} (t+1)2^{-t}t!^{-1}-\sum_{t=q-1}^{\infty} (t+1)2^{-t}t!^{-1}2^{q}q! = \frac{3}{2}e^{1/2} 2^q q! - \sum_{t=q-1}^{\infty} (t+1)2^{-t}t!^{-1}2^{q}q! \\ 
& <  & \frac{3}{2}e^{1/2} 2^q q! - ((q-1)+1)2^{-(q-1)}(q-1)!^{-1}2^{q}q! =\frac{3}{2}e^{1/2} 2^q q! - 2q^2,
 \end{eqnarray*}
where the first equality follows by letting $t=q+1-s$. Comparing the upper and lower bounds for $m$, we get $n \leq 3e^{1/2} 2^q q! - 2q^2 + 6q $. Hence, $f(3,q) < 3e^{1/2} 2^q q!$ for $q>3$.
\end{proof}

With some additional work, one can improve the upper bound in this theorem by an asymptotic factor of $2/3$. In the proof described above, we obtained the interval $J'$ from $J$ by collapsing any instances of $xx$ to $x$. In the worst case, where every letter appears twice, this may cause our interval to shrink by a factor of $1/2$. However, by being more careful, one can get a bound which reflects the fact that one typically needs to collapse adjacent letters only half the time. We suspect that the bound which results from applying this idea may be optimal.

\begin{question}
Prove or disprove that
$$f(3,q)=(2e^{1/2}+o(1))2^q q!$$ 
\end{question}

We next present a lower bound construction, drawing on ideas used in the construction of de Bruijn sequences (see, for example,~\cite{BP07} or~\cite{DG12}), which gives $f(3,q) > 2q!+q-1$. This bound is off from the actual value by a factor $\Theta(2^q)$, but, as we will see below, may be modified to recover this missing factor. 

Say that a word over a $q$-letter alphabet has property $P$ if any two instances of the same letter have distance at least $q-1$ and all intervals of length $q$ are distinct. It is easy to check that any word with property $P$ avoids the Zimin word $Z_3$. Indeed, if there is a copy $xyxzxyx$ of $Z_3$ of minimal length, then the $x$ consists of a single letter. Then $y$ has to consist of at least $q-2$ letters as any two instances of $x$ are at distance at least $q-1$. As we get $xyx$ twice, this implies that there are two identical intervals of length $q$, contradicting property $P$. 

We next prove that the length of the longest word with property $P$ is $2q!+q-1$ and hence $f(3,q) > 2 q!+q-1$. Indeed, it suffices to construct a word with property $P$ of length $2q!+q-1$ as any such word contains each of the $2q!$ possible intervals of length $q$ exactly once and by the pigeonhole principle it follows that this is the longest possible length of such a word.

Construct a directed graph $D$ on the $q!$ words of length $q-1$ over a $q$-letter alphabet that have distinct letters, where an edge is directed from vertex $u$ to vertex  $v$ if the last $q-2$ letters of $u$ are the first $q-2$ letters of $v$. Each vertex of this directed graph has indegree $2$ and outdegree $2$. Thus, $D$ has $2q!$ edges. We claim that this directed graph is strongly connected, that is, it is possible to follow a directed path from any vertex to any other vertex.

\begin{claim}\label{claimstronglyconnected} The directed graph $D$ is strongly connected. 
\end{claim}

\begin{proof}
It suffices to show that there is a walk in $D$ from any vertex $u$ to any other vertex $v$. By symmetry in the letters, we can assume $u$ is the word $12\cdots (q-1)$. For vertices $v$ which correspond to a permutation of $\{1,2,\ldots,q-1\}$, it will suffice to be able to get to any adjacent transposition of $u$, since adjacent transpositions generate the group of permutations. Thus, we simply need to get from $u=12\cdots (q-1)$ to $v=12\cdots (i+1)i \cdots (q-1)$, which is the same as $u$ except $i$ and $i+1$ have switched places. We can do this by considering the 
word formed by concatenating $u$, then the single letter $q$, and then $v$. By considering successive intervals of length $q-1$ from this word of length $2q-1$, we find a walk of length $q$ from $u$ to $v$ in $D$. We also need to show how to get from the word $u=12\ldots (q-1)$ to another word $v$ which doesn't have the same set of $q-1$ letters. Suppose, therefore, that $v$ has letter $q$ and does not have letter $i$. Since we can get from any vertex to any permutation of its letters, we can assume $v$ is the same as $u$ but with $i$ replaced by $q$. But, by concatenating $u$ and $v$, we have a walk from $u$ to $v$ in $D$ of length $q-1$, completing the proof. 
\end{proof}

As the directed graph has equal indegree and outdegree at each vertex and is strongly connected, it is Eulerian, that is, there exists an Eulerian tour covering all the edges. If we form a word by starting with the word of the first vertex and adding one letter at a time for each edge as we walk along the Eulerian tour, this gives a word of the desired length with property $P$. 

We now improve this argument to give a bound which is within a constant factor of the upper bound.

\begin{theorem}
For $q \geq 5$, $f(3,q) > \frac{3}{4}2^q q!+2q-4$.  
\end{theorem}

\begin{proof}
Consider the directed graph $G$ on $2^{q-3} q!$ vertices, where each vertex is formed from a word of length $q-1$ with distinct letters by replacing each internal letter $x$ with $x$ or $xx$. Notice that the vertices are words of length somewhere between $q-1$ and $2q-4$. We place an edge from vertex $u$ to vertex $v$ if the last $q-2$ distinct letters of $u$ is the same as  the first $q-2$ distinct letters from $v$ (this is without repetition of letters) and the subword of $u$ starting at the third distinct letter of $u$ and ending at the second to last letter of $u$ is the same as the subword of $v$ consisting of its second letter to its third to last distinct letter (this is with repetition of letters). For example, if $q=5$ with alphabet $\{a,b,c,d,e\}$, then the outneighbors of vertex $abbccd$ are $bccde$, $bccdde$, $bccda$, and $bccdda$. Each vertex of the directed graph $G$ has indegree $4$ and outdegree $4$, so the number of edges of $G$ is  
 $4 \cdot 2^{q-3} q!=2^{q-1}q!$. 
 
 A slight modification of Claim \ref{claimstronglyconnected} shows that this directed graph is strongly connected. Indeed, the only substantial difference is that we also need to be able to get from one vertex $v$ to the vertex $v'$ which is the same word as $v$ except a single internal letter $x$ that appears by itself in $v$ is replaced by $xx$ or vice versa. But, by concatenating $v$ and $v'$, we get a walk in our directed graph from $v$ to $v'$ of length $q-1$.  As the directed graph is strongly connected and each vertex has equal indegree and outdegree, it is Eulerian, and there is an Eulerian tour starting with a longest vertex (which corresponds to using each of the $q-3$ internal letters twice) covering all of the edges. This Eulerian tour gives rise to a word of the desired length that avoids $Z_3$. Indeed, it avoids $Z_3$ as otherwise we would have two identical copies of $Z_2$, each giving rise to the same edges of $G$ in the Eulerian tour, contradicting the fact that each edge is used exactly once. Furthermore, each vertex has two outgoing edges which add one letter to the end of the word and two outgoing edges which add two letters to the end of the word. This gives an average of $1.5$ letters per edge, after the initial vertex of $2q-4$ letters, giving a total length of $1.5 \cdot 2^{q-1}q!+2q-4=3 \cdot 2^{q-2} q!+2q-4$.
 \end{proof}

\section{Concluding remarks}

{\bf Explicit constructions for $f(3,q)$.}
Our first proof that $f(3,q) \geq q^{q - o(q)}$ is non-constructive, relying upon an application of the Lov\'asz local lemma. However, our second proof, discussed in Section \ref{sec:z3} and giving a bound which is tight to within a constant, can be made algorithmic, constructing the required $Z_3$-free word in time polynomial in its length. Indeed, this proof boils down to constructing an Eulerian tour in an Eulerian directed graph and it is well known that this can be done efficiently. 

Another, stronger notion of explicitness asks that each letter of the word can be computed in time polynomial in $q$. We describe below another construction of a word of length $q^{q-o(q)}$ which is $Z_3$-free and explicit in this sense. This construction is similar to the random construction used in the proof of Theorem \ref{thm:3LLL}, except that the permutation of $L_j$ used on the interval $I_{it+j}$ is now defined explicitly instead of randomly. 

Split the alphabet arbitrarily into $t = \log q$ parts $L_1, L_2, \dots, L_t$, each of size $S = \frac{q}{\log q}$. Let $p_1,\ldots,p_t$ denote the first $t$ primes and $r_j=S!-p_j$ for $1 \leq j \leq t$. Writing $R = S!$, we have $R>r_1>r_2>\ldots>r_t=R-o(R)$ and each pair $r_i,r_j$ with $i > j$ is relatively prime. We construct a word of length $N=tS\prod_{i=2}^{t} r_j = q^{q-o(q)}$, consisting of $N/S$ intervals $I_k$ of length $S$.  For $1 \leq j \leq t$, delete the lexicographic last $p_j$ permutations of $L_j$, keeping the remaining $r_j=S!-p_j$ permutations of $L_j$. With period $r_j$, we use these $r_j$ permutations in lexicographic order to fill the intervals $I_{it+j}$. For this word to contain a copy of $Z_3$, it must contain two identical subwords, each consisting of $t-1$ consecutive intervals $I_k$ of length $S$. But then the difference in their indices must be $t$ times a multiple of $r_j$ for $t-1$ values of $j \in [t]$. Since the $r_j$ are relatively prime, the difference of the indices must therefore be a multiple of $t\prod_{j \in [t] \setminus \{i\}} r_j \geq t\prod_{j=2}^{t} r_j$. However, as the number of intervals is at most $t\prod_{j=2}^{t} r_j$, there cannot be two such intervals, and we are done. 

\vspace{3mm}
{\bf Random words.}
For $n$ fixed and $q$ tending to infinity, it is possible to show that the threshold length for the appearance of $Z_n$ in a random word over an alphabet of size $q$ is $q^{2^{n-1}-(n+1)/2}$. 
For example, over the English alphabet, with $q = 26$, we will likely find a copy of $Z_3$ in a random word of length $1000$ but the minimum word length needed to guarantee a copy is about $10^{34}$. The proof, which we sketch below, is similar to the birthday paradox. This is easiest to see when $n = 2$, as we are simply looking for a word with repeated letters (with the slight caveat that we don't want these letters to be adjacent). 

To prove the upper bound, we estimate the number of copies of $Z_{n-1}$ where each word is a single letter. The length of each such copy of $Z_{n-1}$ is $2^{n-1}-1$ and, as there are $n-1$ variables $x_i$, we see that the probability a random word of length $2^{n-1}-1$ is a copy of $Z_{n-1}$ is $q^{(n-1)-(2^{n-1}-1)}$. Therefore, if we take a random word of length $N$, we expect roughly $N q^{(n-1)-(2^{n-1}-1)}$ such copies of $Z_{n-1}$. Furthermore, the number of copies will be concentrated around this value and almost all of them will be disjoint and separated by at least one letter. We have $D=q^{n-1}$ possible copies of $Z_{n-1}$ of this type 
(for comparison to the birthday paradox, think of $D$ as the number of days in a year) 
and once we get about $D^{1/2}=q^{(n-1)/2}$ of these short copies of $Z_{n-1}$, we will likely get two that are the same, giving a copy of $Z_n$. So we want $N$ with $Nq^{(n-1)-(2^{n-1}-1)} = D^{1/2} = q^{(n-1)/2}$ and, hence, $N = q^{2^{n-1}-(n+1)/2}$. 

The lower bound is a union bound over all possible $Z_n$ (most of them are very unlikely, as the $Z_{n-1}$ are so long that getting repeats of the same long word is incredibly unlikely). 
In fact, there is even a hitting time result (think of building a word one letter at a time here, adding letters at the end) saying that $Z_n$ almost surely appears at the same time when you first find two identical copies of $Z_{n-1}$, each of length $2^{n-1} - 1$. 

\vspace{3mm}
{\bf $q$-unavoidability.}
Recall that a pattern $P$ is $q$-unavoidable if every sufficiently long word over a $q$-letter alphabet contains a copy of $P$. Though the results of Zimin and Bean, Ehrenfeucht and McNulty completely determine those patterns which are $q$-unavoidable for all $q$, much less is known about the patterns which are $q$-unavoidable for some $q$. In particular, given $q$, one may ask whether there is a pattern which is $q$-unavoidable but $(q+1)$-avoidable. Words with this property are known for $q =2, 3$ and $4$, but it is an open problem to construct such words for $q \geq 5$. To give some indication of the difficulty, we note that the pattern constructed by Clark~\cite{C06} which is $4$-unavoidable but $5$-avoidable is $P = abvacwbaxbcycdazdcd$, which admits no obvious generalisation. In light of such difficulties, we believe that any further progress on understanding those patterns which are $q$-unavoidable for some but not all $q$ would be interesting.

\vspace{3mm}
\noindent
{\bf Note added in proof.} After this paper was completed, we learned that a variant of our Theorem~\ref{thm:main2} was obtained simultaneously and independently by Carayol and G\"oller~\cite{CG17}.  It is also worth noting that the results of Section~\ref{sec:local} give an affirmative answer to a question raised in their paper, namely, whether the probabilistic method can be used to give a tower-type lower bound for the function $f(n,q)$.

\end{document}